\author{\Large{Damanvir Singh Binner 
}}
\documentclass[12pt]{article}
\usepackage[margin=1in]{geometry}
\usepackage[usenames]{xcolor}
\usepackage{amssymb}
\usepackage{relsize}
\usepackage{amsmath}
\usepackage{amsthm}
\usepackage{amsfonts}
\usepackage{amscd}
\usepackage{graphicx}
\usepackage{hyperref}
\usepackage{tikz-cd}
\usepackage{tikz}

\begin{document}

\newcommand{\D}[1]{{\bf \color{red} #1}}
\newcommand{\M}[1]{{\bf \color{magenta} #1}}
\newcommand{\K}[1]{{\bf \color{violet} #1}}

\theoremstyle{plain}
\newtheorem{theorem}{Theorem}
\newtheorem{corollary}[theorem]{Corollary}
\newtheorem{lemma}[theorem]{Lemma}
\newtheorem{proposition}[theorem]{Proposition}
\newtheorem{question}[theorem]{Question}

\theoremstyle{definition}
\newtheorem{definition}[theorem]{Definition}
\newtheorem{example}[theorem]{Example}
\newtheorem{conjecture}[theorem]{Conjecture}

\theoremstyle{remark}
\newtheorem{remark}[theorem]{Remark}

\title{\Large{On $k$-measures and Durfee squares of partitions}}
\date{}
\maketitle
\begin{center}
\vspace*{-8mm}
\large{Department of Mathematics \\
Indian Institute of Science Education and Research (IISER) \\
Mohali, Punjab, India \\
damanvirbinnar@iisermohali.ac.in}
\end{center}

\begin{abstract}
Recently, Andrews, Bhattacharjee and Dastidar introduced the concept of $k$-measure of an integer partition, and proved a surprising identity that the number of partitions of $n$ which have $2$-measure $m$ is equal to the number of partitions of $n$ with a Durfee square of side $m$. The authors asked for a bijective proof of this result and also suggested a further exploration of the properties of the number of partitions of $n$ which have $k$-measure $m$ for $k \geq 3$. In this note, we perform these tasks. That is, we obtain a short combinatorial proof of the result of Andrews, Bhattacharjee and Dastidar, and using this proof, we obtain a natural generalization for $k$-measures. 
\end{abstract}

\section{Introduction}
\label{S1}

Andrews, Bhattacharjee and Dastidar \cite{Andrews1} introduced a new statistic of integer partitions, which they named as the $k$-measure. 

\begin{definition}
The $k$-measure of a partition is the length of the largest subsequence of parts in the partition in which the difference between any two consecutive parts of the subsequence is at least $k$.
\end{definition}

Recall that the Durfee square of a partition is the largest square that can be constructed in the Ferrers diagram of the partition beginning from the top left corner. Andrews, Bhattacharjee and Dastidar \cite{Andrews1} found a deep connection between these two statistics of a partition as described in the theorem below.

\begin{theorem}[Andrews, Bhattacharjee and Dastidar (2022)]
\label{Wonderful}
The number of partitions of $n$ with $2$-measure m equals the number of partitions of $n$ with Durfee square of side $m$.
\end{theorem}

The authors proved Theorem \ref{Wonderful} using $q$-series analysis. They concluded the paper by noting that a theorem as simply stated as Theorem \ref{Wonderful} should definitely have a bijective proof. They also suggested a further study of the properties of $k$-measures of partitions for $k > 2$. In a subsequent work, Andrews, Chern and Li \cite{Andrews2} used generalized Heine transformations to establish some trivariate generating function identities counting both the length and the k-measure for partitions and distinct partitions, respectively. As a corollary of their result, they obtained the following refinement of Theorem \ref{Wonderful}.

\begin{theorem}[Andrews, Chern and Li]
\label{Wonderful2}
The number of partitions of $n$ with $l$ parts and $2$-measure m equals the number of partitions of $n$ with $l$ parts and Durfee square of side $m$.
\end{theorem}

As another corollary of their result, they got the following result for partitions with distinct odd parts. They defined $l(\lambda)$ to be the number of parts of $\lambda$ and $\mu_2(\lambda)$ to be the $2$-measure of $\lambda$. 

\begin{theorem}[Andrews, Chern and Li]
\label{Wonderful3}
The excess of the number of partitions $\lambda$ of $n$ with $l(\lambda) + \mu_2(\lambda)$ even over those with $l(\lambda) + \mu_2(\lambda)$ odd equals the number of partitions of $n$ into distinct odd parts.
\end{theorem}

Also refer to \cite{Lin} for some recent refinements of Theorem \ref{Wonderful} and an alternate $q$-series proof of Theorem \ref{Wonderful2}. In Section \ref{S2}, we obtain a short combinatorial proof of Theorem \ref{Wonderful}. In fact, our proof also proves the more general result of Theorem \ref{Wonderful2}. In Section \ref{S3}, we use the ideas in this proof to generalize Theorem \ref{Wonderful} for $k$-measures.

\section{Proof of Theorem \ref{Wonderful}}
\label{S2}

Let 

\begin{itemize}
\item $a_m(n)$ denotes the number of partitions of $n$ with $2$-measure $m$.
\item $b_m(n)$ denotes the number of partitions of $n$ with Durfee square of side $m$ by $b_m(n)$.
\end{itemize}

 Thus, Theorem \ref{Wonderful} asserts that $a_m(n) = b_m(n)$ for all $m$ and $n$. We 
begin by noting that $$a_m(n) = c_m(n) - c_{m+1}(n),$$ and  $$b_m(n) = d_m(n) - d_{m+1}(n),$$ where $c_m(n)$ and $d_m(n)$ are defined as follows.

\begin{itemize}
\item $C_m(n)$ denotes the set of partitions of $n$ in which there exists a subsequence of length $m$ of parts in the partition in which the difference between any two consecutive parts of the subsequence is at least $2$.
\item $D_m(n)$ denotes the set of partitions of $n$ which have at least $m$ parts greater than or equal to $m$. 
\item $c_m(n) = |C_m(n)|$.
\item $d_m(n) = |D_m(n)|$. 
\end{itemize}

From here, it immediately follows that $a_m(n) = b_m(n)$ for all $m$ and $n$ if and only if $c_m(n)=d_m(n)$ for all $m$ and $n$. That is Theorem \ref{Wonderful} is equivalent to Theorem \ref{Easy} described below, for which we provide a short bijective proof. 

\begin{theorem}
\label{Easy}
The number of partitions of $n$ in which there exists a subsequence of length $m$ of parts in the partition in which the difference between any two consecutive parts of the subsequence is at least $2$ is equal to the number of partitions of $n$ which have at least $m$ parts greater than or equal to $m$. 
\end{theorem}

\begin{proof} 
We construct a bijection $\phi$ between the sets $C_m(n)$ and $D_m(n)$. Prior to that, we describe the motivation behind this bijection with the help of an example. Suppose $m=5$. Then, we are given a subsequence $(\lambda_1, \lambda_2, \lambda_3, \lambda_4, \lambda_5)$ with $$\lambda_1 \geq \lambda_2 + 2, \lambda_2 \geq \lambda_3 + 2, \lambda_3 \geq \lambda_4 + 2 , \lambda_4 \geq \lambda_5 + 2.$$ Thus, in particular, we have $$\lambda_1 \geq 9, \lambda_2 \geq 7, \lambda_3 \geq 5, \lambda_4 \geq 3, \lambda_5 \geq 1.$$ We need to map this to a partition with all parts greater than or equal to $5$. Note that the average of these lower bounds $9,7,5,3,1$ is equal to $5$. Therefore, we modify all the members of the subsequence $(\lambda_1, \lambda_2, \lambda_3, \lambda_4, \lambda_5)$ in such a way that their lower bound sequence $9,7,5,3,1$ becomes the average sequence $5,5,5,5,5$. To do this, we basically reduce the first number by $4$, reduce the second number by $2$, keep the third number unchanged, increase the fourth number by $2$ and increase the last number by $4$. That is, we map the subsequence $(\lambda_1, \lambda_2, \lambda_3, \lambda_4, \lambda_5)$ to $(\lambda_1-4, \lambda_2-2, \lambda_3, \lambda_4+2, \lambda_5+4)$. Clearly, each component of this vector is a number greater than or equal to $5$. To make the pattern even more clear, we can also express the above vector as $(\lambda_1-4, \lambda_2-2, \lambda_3 - 0, \lambda_4 -(-2), \lambda_5-(-4))$. Now, we describe the bijection $\phi$ for a general $m$.

 Suppose we have a partition $\lambda \in C_m(n)$. Further, suppose $(\lambda_1, \lambda_2, \cdots, \lambda_m)$ be a subsequence of parts of $\lambda$ in which the difference between any two consecutive parts of the subsequence is at least $2$. That is, $\lambda_i \geq \lambda_{i+1} + 2$ for all $i$. In particular, we have $\lambda_i \geq 1 + 2(m-i)$ for all $1 \leq i \leq m$. We define the map $\phi$ to be such that the elements of $\lambda$ not in our chosen subsequence $(\lambda_1, \lambda_2, \cdots, \lambda_m)$ are left unchanged, and $(\lambda_1, \lambda_2, \cdots, \lambda_m)$ is mapped under $\phi$ to 
\begin{equation}
\label{Balance2}
\Big(\lambda_1 - (m-1), \lambda_2 - (m-3), \cdots, \lambda_i - (m-(2i-1)), \cdots, \lambda_{m-1} + (m-3), \lambda_m + (m-1) \Big) .
\end{equation}

We verify that the resultant partition is indeed a member of $D_{m}(n)$. For that, we note two things. First, we have $$ \sum_{i = 1}^m \Big(\lambda_i - (m-(2i-1))\Big) =  \sum_{i = 1}^m \lambda_i.$$ That is, the other terms with the positive and negative signs cancel out.  Secondly, using  $\lambda_i \geq 1 + 2(m-i)$ for all $1 \leq i \leq m$, one easily observes that $$\lambda_i - (m-(2i-1)) \geq m$$ for all $1 \leq i \leq m$. Thus, all the members of the vector in \eqref{Balance2} are greater than or equal to $m$, and therefore the resultant partition is indeed a member of $D_{m}(n)$. 
Next, we show that the map $\phi$ is invertible by constructing the inverse map $\psi$. Though $\psi$ is easy to predict, we describe it below in some detail.

Suppose we have a partition $\pi \in D_m(n)$. There exist some parts $\pi_1, \pi_2, \cdots, \pi_m$ of $\pi$ such that $$ \pi_1 \geq \pi_2 \cdots \geq \pi_m \geq m. $$ We define the map $\psi$ to be such that the elements of $\pi$ other than $(\pi_1, \pi_2, \cdots, \pi_m)$ are left unchanged, and $(\pi_1, \pi_2, \cdots, \pi_m)$ is mapped under $\psi$ to 
\begin{multline*}
\Big(\pi_1 + (m-1), \pi_2 + (m-3), \cdots,  \pi_i + (m-(2i-1)), \cdots, \pi_{m-1} - (m-3), \pi_m - (m-1) \Big) .
\end{multline*}

It is straightforward to verify that consecutive members of the above vector differ by at least $2$, and thus the resultant partition is indeed a member of $C_m(n)$. Finally, it is easy to check that the maps $\phi$ and $\psi$ are indeed inverses of each other, completing the proof of Theorem \ref{Easy}, and thus also of Theorem \ref{Wonderful}. 
\end{proof}

\begin{remark}
Since the maps $\phi$ and $\psi$ preserve the number of parts, our proof also gives a combinatorial proof of Theorem \ref{Wonderful2}.
\end{remark}

\section{Generalization of Theorem \ref{Wonderful} for $k$-measures}
\label{S3}

It turns out that Theorem \ref{Easy} is easy to generalize by appropriately modifying the maps $\phi$ and $\psi$. We provide all the details for the sake of completeness. We denote the floor and ceiling functions of $x$ by $\lfloor x \rfloor$ and $\lceil x \rceil$ respectively. Note that for any natural number $m$, $ \left \lfloor \frac{m}{2} \right \rfloor + \left \lceil \frac{m}{2} \right \rceil = m$. This fact will be used frequently in the proof of the next theorem.

\begin{theorem}
\label{General}
The number of partitions of $n$ in which there exists a subsequence of length $m$ of parts in the partition in which the difference between any two consecutive parts of the subsequence is at least $k$ is equal to the number of partitions of $n$ which have at least $\left \lfloor \frac{m}{2} \right \rfloor $ parts greater than or equal to $1 + \left \lceil \frac{k(m-1)}{2} \right \rceil $, and an additional at least $\left \lceil \frac{m}{2} \right \rceil $ parts greater than or equal to $1 + \left \lfloor \frac{k(m-1)}{2} \right \rfloor $.
\end{theorem}

\begin{proof}
Let 
\begin{itemize}
\item $C_{k,m}(n)$ denotes the set of partitions of $n$ in which there exists a subsequence of length $m$ of parts in the partition in which the difference between any two consecutive parts of the subsequence is at least $k$.
\item $D_{k,m}(n)$ denotes the set of partitions of $n$ which have at least $\left \lfloor \frac{m}{2} \right \rfloor $ parts greater than or equal to $1 + \left \lceil \frac{k(m-1)}{2} \right \rceil $, and an additional at least $\left \lceil \frac{m}{2} \right \rceil $ parts greater than or equal to $1 + \left \lfloor \frac{k(m-1)}{2} \right \rfloor $.
\end{itemize}

We construct a bijection $\phi'$ between $C_{k,m}(n)$ and $D_{k,m}(n)$. Prior to that, we describe the motivation behind this bijection. 

Suppose we have a partition $\lambda \in C_{k,m}(n)$. Further, suppose $(\lambda_1, \lambda_2, \cdots, \lambda_m)$ be a subsequence of parts of $\lambda$  in which the difference between any two consecutive parts of the subsequence is at least $k$. That is, $\lambda_i \geq \lambda_{i+1} + k$ for all $i$. In particular, we have $\lambda_i \geq 1+k(m-i)$ for all $1 \leq i \leq m$. As suggested by the proof of Theorem \ref{Easy}, we calculate the average of the lower bounds on the $\lambda_i$'s, which comes out to be $1 + \frac{k(m-1)}{2}$. If $k$ is even or $m$ is odd, then $\frac{k(m-1)}{2}$ is an integer and our idea in the proof of Theorem \ref{Easy} can be easily generalized to construct the required bijection. 

However if $k$ is odd and $m$ is even, then the average $1 + \frac{k(m-1)}{2}$ of the lower bounds on the $\lambda_i$'s is not an integer. This problem makes the analysis of this case a little harder. We explain this difficulty by considering an example. Suppose $k=3$ and $m=4$. Then, we are given a subsequence $(\lambda_1, \lambda_2, \lambda_3, \lambda_4)$ with $$\lambda_1 \geq \lambda_2 + 3, \lambda_2 \geq \lambda_3 + 3, \lambda_3 \geq \lambda_4 + 3.$$ Thus, in particular, we have $$\lambda_1 \geq 10, \lambda_2 \geq 7, \lambda_3 \geq 4, \lambda_4 \geq 1.$$ The average of these lower bounds on the $\lambda_i$'s comes out to be $5.5$. Now if we just try to use the previous approach, we would map this subsequence to $ (\lambda_1 - 4.5, \lambda_2-1.5, \lambda_3 + 1.5, \lambda_4 + 4.5)$. However, since the parts of a partition must be integers, we cannot map it like this. Thus, the appropriate modification we make in this case is to map the subsequence to $ (\lambda_1 - 4, \lambda_2-1, \lambda_3 + 1, \lambda_4 + 4)$ instead. Note that in this case, the resultant partition has the property that at least two parts are greater than or equal to $6$ and an additional two parts are greater than or equal to $5$, explaining the strange condition in the definition of $D_{k,m}(n)$. 
Next, we describe the map $\phi'$ for any $k$ and $m$. Using floor and ceiling functions, we will be able to handle together the two cases when $k(m-1)$ is odd or even. 

We define the map $\phi'$ to be such that the elements of $\lambda$ not in our chosen subsequence $(\lambda_1, \lambda_2, \cdots, \lambda_m)$ are left unchanged, and $(\lambda_1, \lambda_2, \cdots, \lambda_m)$ is mapped under $\phi'$ to 
\begin{multline}
\label{Balancek}
\Bigg(\lambda_1 - \left \lfloor \frac{k(m-1)}{2} \right \rfloor, \lambda_2 - \left \lfloor \frac{k(m-3)}{2} \right \rfloor, \cdots, \lambda_{\left \lfloor \frac{m}{2} \right \rfloor} - \left \lfloor \frac{k}{2} \left(m+1-2\left\lfloor \frac{m}{2} \right \rfloor \right) \right \rfloor, \\
 \lambda_{\left \lfloor \frac{m}{2} \right \rfloor + 1} + \left \lfloor \frac{k}{2}  \left(m+1-2\left\lceil \frac{m}{2} \right \rceil \right) \right \rfloor, \lambda_{\left \lfloor \frac{m}{2} \right \rfloor + 2} + \left \lfloor \frac{k}{2}  \left(m+3-2\left\lceil \frac{m}{2} \right \rceil \right) \right \rfloor, \cdots \\
  \cdots, \lambda_{m-1} + \left \lfloor \frac{k(m-3)}{2} \right \rfloor, \lambda_m + \left \lfloor \frac{k(m-1)}{2} \right \rfloor \Bigg) .
\end{multline}

The members of the vector in \eqref{Balancek} can be described compactly as follows. The first $\left \lfloor \frac{m}{2} \right \rfloor$ members can be written as $$ \left\{\lambda_i - \left \lfloor \frac{k(m-(2i-1))}{2} \right \rfloor: 1 \leq i \leq \left \lfloor \frac{m}{2} \right \rfloor \right\}, $$ while the remaining $\left \lceil \frac{m}{2} \right \rceil $ members can be written (beginning from right to left) as $$ \left\{\lambda_{m-i} + \left \lfloor  \frac{k(m-(2i+1)}{2}  \right \rfloor : 0 \leq i < \left \lceil \frac{m}{2} \right \rceil \right\}. $$
 
We prove that the resultant partition is indeed a member of $D_{k,m}(n)$. First, considering two cases based on the parity of $m$, it is an easy exercise to confirm that the sum of the members of the vector in \eqref{Balancek} is equal to the sum of $\lambda_i$'s. That is, the other terms with the positive and negative signs cancel out. Secondly, we show that the first $\left \lfloor \frac{m}{2} \right \rfloor$ members of the the vector in \eqref{Balancek} are greater than or equal to $1 + \left \lceil \frac{k(m-1)}{2} \right \rceil $, and the last $\left \lceil \frac{m}{2} \right \rceil $ members are greater than or equal to $1 + \left \lfloor \frac{k(m-1)}{2} \right \rfloor $. To prove these facts, we crucially use $\lambda_i \geq 1+k(m-i)$ for all $1 \leq i \leq m$. Therefore, for $1 \leq i \leq \left \lfloor \frac{m}{2} \right \rfloor$, we have 

\begin{align*}
\lambda_i - \left \lfloor \frac{k(m-(2i-1))}{2} \right \rfloor &\geq 1+k(m-i) - \left \lfloor \frac{k(m-(2i-1))}{2} \right \rfloor  \\
&= 1+km -  \left \lfloor \frac{k(m+1))}{2} \right \rfloor \\
&= 1+k(m-1) -  \left \lfloor \frac{k(m-1))}{2} \right \rfloor \\
&= 1 + \left \lceil \frac{k(m-1)}{2} \right \rceil.
\end{align*}

Similarly, for $0 \leq i < \left \lceil \frac{m}{2} \right \rceil$, we have 
\begin{align*}
\lambda_{m-i} + \left \lfloor  \frac{k(m-(2i+1))}{2}  \right \rfloor  &\geq 1+ki+ \left \lfloor  \frac{k(m-(2i+1))}{2}  \right \rfloor \\
& = 1 + \left \lfloor  \frac{k(m-1)}{2}  \right \rfloor,
\end{align*}
as required. Thus, the resultant partition is indeed a member of $D_{k,m}(n)$. Next, we show that the map $\phi'$ is invertible by constructing the inverse map $\psi'$. The map $\psi'$ is again easy to guess but we describe it in some detail below. 

Suppose we have a partition $\pi \in D_m(n)$. There exist some parts $\pi_1, \pi_2, \cdots, \pi_m$ of $\pi$ such that $$ \pi_1 \geq \pi_2 \cdots \geq \pi_m \geq m. $$ We define the map $\psi'$ to be such that the elements of $\pi$ other than $(\pi_1, \pi_2, \cdots, \pi_m)$ are left unchanged, and $(\pi_1, \pi_2, \cdots, \pi_m)$ is mapped under $\psi'$ to 
\begin{multline*}
\Bigg(\pi_1 + \left \lfloor \frac{k(m-1)}{2} \right \rfloor, \pi_2 + \left \lfloor \frac{k(m-3)}{2} \right \rfloor, \cdots, \pi_{\left \lfloor \frac{m}{2} \right \rfloor} + \left \lfloor \frac{k}{2} \left(m+1-2\left\lfloor \frac{m}{2} \right \rfloor \right) \right \rfloor, \\
 \pi_{\left \lfloor \frac{m}{2} \right \rfloor + 1} - \left \lfloor \frac{k}{2}  \left(m+1-2\left\lceil \frac{m}{2} \right \rceil \right) \right \rfloor, \pi_{\left \lfloor \frac{m}{2} \right \rfloor + 2} - \left \lfloor \frac{k}{2}  \left(m+3-2\left\lceil \frac{m}{2} \right \rceil \right) \right \rfloor, \cdots \\
 \cdots, \pi_{m-1} - \left \lfloor \frac{k(m-3)}{2} \right \rfloor, \pi_m - \left \lfloor \frac{k(m-1)}{2} \right \rfloor \Bigg) .
\end{multline*}

It is easy to verify that consecutive members of the above vector differ by at least $k$, and thus the resultant partition is indeed a member of $C_{k,m}(n)$. Finally, it is also straightforward to verify that the maps $\phi'$ and $\psi'$ are indeed inverses of each other, completing the proof of Theorem \ref{General}.

\end{proof}

Next, we use Theorem \ref{General} to generalize Theorems \ref{Wonderful} and \ref{Wonderful2}. First, we deduce two immediate corollaries of Theorem \ref{General} that will be helpful to obtain these generalizations. 

\begin{corollary}
\label{EO}
Suppose either $k$ is even or $m$ is odd. Then the number of partitions of $n$ in which there exists a subsequence of length $m$ of parts in the partition in which the difference between any two consecutive parts of the subsequence is at least $k$ is equal to the number of partitions of $n$ which have at least $m$ parts of $1+\frac{k(m-1)}{2}$.
\end{corollary}

\begin{corollary}
\label{OE}
Suppose $k$ is odd and $m$ is even. Then the number of partitions of $n$ in which there exists a subsequence of length $m$ of parts in the partition in which the difference between any two consecutive parts of the subsequence is at least $k$ is equal to the number of partitions of $n$ which have at least $\frac{m}{2}$ parts of $\frac{k(m-1)+1}{2}$, and an additional at least $\frac{m}{2}$ parts of $\frac{k(m-1)+3}{2}$.
\end{corollary}

Based on the results in these corollaries, we define a $(k,m)$-polygon associated to an integer partition. 

\begin{itemize}
\item Suppose $k$ is even or $m$ is odd. Then define the $(k,m)$-polygon of a partition $\pi$ to be the rectangle containing $m$ rows of $1+\frac{k(m-1)}{2}$ nodes (In other words, the rectangle with vertical side $m$ and horizontal side $1+\frac{k(m-1)}{2}$) beginning from the top left corner of the Ferrers diagram of $\pi$. For example, the partition $9+9+8+7+4+3+1$ of $41$ has the following $(4,3)$-polygon.
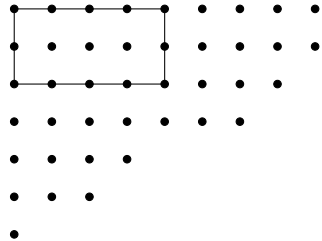
\vspace{1cm}

 \begin{center}
\begin{tikzpicture}

\foreach \x in {0,...,8}
	\filldraw (\x*.5, -5) circle (.5mm);
\foreach \x in {0,...,8}
	\filldraw (\x*.5, -5.5) circle (.5mm);
\foreach \x in {0,...,7}
	\filldraw (\x*.5, -6) circle (.5mm);
\foreach \x in {0,...,6}
	\filldraw (\x*.5, -6.5) circle (.5mm);
\foreach \x in {0,...,3}
	\filldraw (\x*.5, -7) circle (.5mm);.
\foreach \x in {0,...,2}
	\filldraw (\x*.5, -7.5) circle (.5mm);.
\foreach \x in {0}
	\filldraw (\x*.5, -8) circle (.5mm);.
	
\draw[solid,black] (0,-5)-- (0,-6);
\draw[solid,black] (0,-6)-- (2,-6);
\draw[solid,black] (2,-5)-- (2,-6);
\draw[solid,black] (0,-5)-- (2,-5);
\end{tikzpicture}
\end{center}

\vspace{1cm}

\item Suppose $k$ is odd and $m$ is even. Then define the $(k,m)$-polygon of a partition $\pi$ to be the polygon containing $\frac{m}{2}$ rows of $\frac{k(m-1)+3}{2}$ nodes and $\frac{m}{2}$ rows of $\frac{k(m-1)+1}{2}$ nodes beginning from the top left corner of the Ferrers diagram of $\pi$. For example, the partition $9+9+8+7+4+3+1$ of $41$ has the following $(3,4)$-polygon.
\end{itemize}

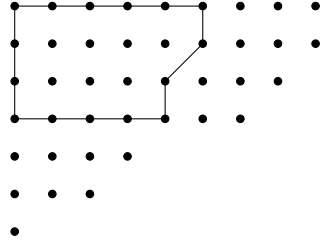
\vspace{1cm}

 \begin{center}
\begin{tikzpicture}

\foreach \x in {0,...,8}
	\filldraw (\x*.5, -5) circle (.5mm);
\foreach \x in {0,...,8}
	\filldraw (\x*.5, -5.5) circle (.5mm);
\foreach \x in {0,...,7}
	\filldraw (\x*.5, -6) circle (.5mm);
\foreach \x in {0,...,6}
	\filldraw (\x*.5, -6.5) circle (.5mm);
\foreach \x in {0,...,3}
	\filldraw (\x*.5, -7) circle (.5mm);.
\foreach \x in {0,...,2}
	\filldraw (\x*.5, -7.5) circle (.5mm);.
\foreach \x in {0}
	\filldraw (\x*.5, -8) circle (.5mm);.
	
\draw[solid,black] (0,-5)-- (0,-6.5);
\draw[solid,black] (0,-5)-- (2.5,-5);
\draw[solid,black] (0,-6.5)-- (2,-6.5);
\draw[solid,black] (2,-6.5)-- (2,-6);
\draw[solid,black] (2,-6)-- (2.5,-5.5);
\draw[solid,black] (2.5,-5)-- (2.5,-5.5);
\end{tikzpicture}
\end{center}

\vspace{1cm}

Based on this definition, we can rewrite Corollaries \ref{EO} and \ref{OE} together as follows.

\begin{corollary}
\label{Final}
Suppose $k$ is odd and $m$ is even. Then the number of partitions of $n$ in which there exists a subsequence of length $m$ of parts in the partition in which the difference between any two consecutive parts of the subsequence is at least $k$ is equal to the number of partitions of $n$ whose Ferrers diagram contains its $(k,m)$-polygon.
\end{corollary}

It is easy to observe the following properties of the $(k,m)$-polygons of partitions. 

\begin{itemize}
\item For a given $k$ and a partition $\pi$, the $(k,m)$-polygon of $\pi$ is strictly contained in the $(k,m+1)$-polygon of $\pi$ for any $m$, irrespective of the parities of $k$ and $m$. 
\item For a given $k$ and a partition $\pi$, there are only finitely many values of $m$ such that the Ferrers diagram of $\pi$ contains the $(k,m)$-polygon of $\pi$. 
\end{itemize}

From these observations, it is obvious that for a given $k$ and a partition $\pi$, there exists a largest value of $m$ such that the Ferrers diagram of $\pi$ contains the $(k,m)$-polygon of $\pi$. We say that $\pi$ has a $(k,m)$-Durfee polygon.
Then from Corollary \ref{Final}, the following result immediately follows.

\begin{theorem}
\label{FinalMeasure}
The number of partitions of $n$ with $k$-measure $m$ is equal to the number of partitions of $n$ with $(k,m)$-Durfee polygon. 
\end{theorem}

\begin{remark}
\label{Durfee}
Note that for $k=2$, the $(2,m)$-polygonal of $\pi$ is a square of side $m$, and thus for a partition $\pi$ to have $(2,m)$-Durfee polygon is the same thing as $\pi$ having a Durfee square of side $m$. Thus, Theorem \ref{FinalMeasure} is a generalization of Theorem \ref{Wonderful}. 
\end{remark}

\begin{remark}
Since the maps $\phi'$ and $\psi'$ preserve the number of parts, our proof also gives a generalization of Theorem \ref{Wonderful2}. That is, the number of partitions of $n$ with $l$ parts and $k$-measure $m$ is equal to the number of partitions of $n$ with $l$ parts and $(k,m)$-Durfee polygon. 
\end{remark}

\begin{remark}
For even $k$, the $(k,m)$-polygons are rectangles irrespective of the value of $m$. Thus, for an even $k$, Theorem \ref{FinalMeasure} provides a very simple looking generalization of Theorem \ref{Wonderful}. For example, the number of partitions of $n$ with $4$-measure $m$ is equal to the number of partitions for which the largest value of $j$ such that the Ferrers diagram of the partition contains a $j \times (2j-1)$ rectangle equals $m$. Similarly, the number of partitions of $n$ with $6$-measure $m$ is equal to the number of partitions for which the largest value of $j$ such that the Ferrers diagram of the partition contains a $j \times (3j-2)$ rectangle equals $m$. For the case when $k$ is odd, the situation is relatively complex as $(k,m)$-polygons can be rectangles or not depending on the parity of $m$. 
\end{remark}

These results seem to yield interesting properties even in the case $k=1$. For example, substituting $k=1$ in Theorem \ref{General}, we get the following result. 

\begin{corollary}
\label{k=1}
The number of partitions of $n$ which have at least $m$ distinct parts is equal to the number of partitions of $n$ which have at least $\left \lfloor \frac{m}{2} \right \rfloor $ parts greater than or equal to $ \left \lceil \frac{m+1)}{2} \right \rceil $, and an additional at least $\left \lceil \frac{m}{2} \right \rceil $ parts greater than or equal to $ \left \lfloor \frac{m+1}{2} \right \rfloor $.
\end{corollary}

That is for an odd number $m$, the number of partitions of $n$ which have at least $m$ distinct parts is equal to the number of partitions of $n$ which have at least $m$ parts greater than or equal to $\frac{m+1}{2}$. For example, the number of partitions of $n$ which have at least $7$ distinct parts is equal to the number of partitions of $n$ which have at least $7$ parts greater than or equal to $4$. 

Similarly, for an even number $m$, the number of partitions of $n$ which have at least $m$ distinct parts is equal to the number of partitions of $n$ which have at least $\frac{m}{2}$ parts greater than or equal to $\frac{m}{2}+1$, and an additional at least $\frac{m}{2}$ parts greater than or equal to $\frac{m}{2}$. For example, the number of partitions of $n$ which have at least $6$ distinct parts is equal to the number of partitions of $n$ which have at least $3$ parts greater than or equal to $4$, and an additional at least $3$ parts greater than or equal to $3$. 

Finally, substituting $k=1$ in Corollary \ref{Final} gives the following result.

\begin{corollary}
\label{2;k=1}
The number of partitions of $n$ with $m$ distinct parts is equal to the number of partitions of $n$ with $(1,m)$-Durfee polygon. 
\end{corollary}

\section{Concluding Remarks}

It will be very interesting to see if one could utilize the ideas in this paper to obtain a combinatorial proof and possibly generalizations of Theorem \ref{Wonderful3}. 

\section{Acknowledgement}
  
 The author acknowledges the support of IISER Mohali for providing research facilities and fellowship.

\end{document}